\documentclass[12pt,fleqn]{amsart}

\usepackage{amssymb}

\usepackage{color}
\usepackage{enumerate}
\newtheorem{theorem}{Theorem}[section]
\newtheorem{lemma}[theorem]{Lemma}

\newtheorem{conjecture}[theorem]{Conjecture}
\newtheorem{corollary}[theorem]{Corollary}

\theoremstyle{definition}
\newtheorem{definition}[theorem]{Definition}

\theoremstyle{remark}
\newtheorem{remark}[theorem]{Remark}
\numberwithin{equation}{section}

\newcommand{\fC}{\mathfrak C}
\newcommand{\fcc}{{\rm FCC}}

\newcommand{\cC}{\mathcal C}

\newcommand{\cU}{\mathcal U}
\newcommand{\cV}{\mathcal V}

\newcommand{\fd}{\protect{\rm free}$-${\rm dim}}
\newcommand{\C}{\mathcal C}

\newcommand{\II}{\mathcal I}

\newcommand{\PP}{\mathcal P}

\newcommand{\N}{\mathbb N}

\newcommand{\fA}{\mathfrak A}
\newcommand{\fB}{\mathfrak B}

\newcommand{\cG}{\mathcal G}
\newcommand{\cF}{\mathcal F}
\newcommand{\cH}{\mathcal H}

\newcommand{\sub}{\subseteq}
\newcommand{\eps}{\varepsilon}
\newcommand{\er}{\mathbb R}
\newcommand{\en}{\mathbb N}

\def\At{\operatorname{Atom}}

\begin{document}
\baselineskip=17pt

\title[Marde\v{s}i\'c conjecture]{Marde\v{s}i\'c conjecture and free products\\ of Boolean algebras}

\author[G.\ Mart\'{\i}nez-Cervantes]{Gonzalo Mart\'{\i}nez-Cervantes}
\address{Departamento de Matem\'{a}ticas\\
Facultad de Matem\'{a}ticas\\ Universidad de Murcia\\ 30100 Espinardo, Murcia\\
Spain}
 \email{gonzalo.martinez2@um.es}
\author[G. Plebanek]{Grzegorz Plebanek}
\address{Instytut Matematyczny\\ Uniwersytet Wroc\l awski\\ Pl.\ Grunwaldzki 2/4\\
50-384 Wroc\-\l aw\\ Poland}
\email{grzes@math.uni.wroc.pl}

\thanks{
The first author was partially supported by  the research project 19275/PI/14 funded by Fundaci\'{o}n S\'{e}neca -- Agencia de Ciencia y Tecnolog\'{i}a de la Regi\'{o}n de Murcia within the framework of PCTIRM 2011-2014 and by Ministerio de Econom\'{i}a y Competitividad and FEDER (project MTM2014-54182-P).\\
The research done during the second's author stay at Facultad de Matem\'aticas, Universidad de Murcia, 
supported by { Fundaci\'on S\'eneca} -- Agencia de Ciencia y Tecnolog\'{i}a de la Regi\'{o}n de Murcia,
 through its Re\-gio\-nal Programme {\em Jim\'enez de la Espada}.}
\subjclass[2010]{Primary 54F05, 54F45; secondary 03G05, 06E15.}

\keywords{Linearly ordered topological space, interval algebra, free product.}

\date{}

\begin{abstract}
We show that for every $d\ge 1$, if $L_1,\ldots, L_d$ are linearly ordered compact spaces and there is a continuous surjection 
\[ L_1\times L_2\times \dots\times L_d\to K_1\times K_2\times\ldots\times K_{d}\times K_{d+1},\]
where all the spaces $K_i$ are infinite, then $K_i, K_j$ are metrizable for some $1\le i<j\le d+1$.
This answers a problem posed by Marde\v{s}i\'c  \cite{Mard70, Mard15}.

We present some related results  on Boolean algebras not containing free products with too many uncountable factors.
In particular, we answer a problem on initial chain algebras that was posed in \cite{BaurHeindorf97}.
\end{abstract}

\maketitle
\section{Introduction}\label{i}

The classical Peano curve demonstrates that the unit interval, a metrizable linearly ordered compact space, may be 
continuously mapped onto its square. This cannot happen for compact lines that are not metrizable: 
Treybig  \cite{Trey64} and Ward  \cite{Ward} proved that if a 
product of two infinite compact spaces is a  continuous image of a linearly ordered compact space then such a product is necessarily metrizable (a simpler proof of this result can be found in \cite{BDK81}).  

Subsequently, S. Marde\v{s}i\'c partially generalized the result of Treybig to larger number of factors. He    proved in \cite{Mard70} that whenever a product of $d$ linearly ordered compact spaces can be mapped onto a product of $d+s$ {\em separable} infinite compact spaces $K_1, \dots K_{d+s}$ with $s\geq 1$,  then there are at least $s+1$ metrizable factors $K_j$. In the same paper he raised a question, if the separability assumption may be dropped,
stating the following conjecture  (see also the survey paper    \cite{Mard15}):

\begin{conjecture}\label{i:1}
If a product of $d$ linearly ordered compact spaces can be mapped onto a product of $d+s$ infinite compact spaces $K_1, K_2, \dots, K_{d+s}$ with $s\geq1$, then there are at least $s+1$ metrizable factors $K_j$.
\end{conjecture}

Let us note that  Conjecture \ref{i:1}, in its full generality,  easily follows from its particular case, when  $s=1$.
Indeed, take $K_1, \dots, K_{d+s}$ infinite compact spaces with at most $s$ metri\-za\-ble factors; 
without loss of generality, let $K_1, \dots, K_{d}$ be nonmetrizable. Then
\[ K_1  \times \dots \times K_{d+s} = K_1 \times \dots \times K_{d} \times \big( K_{d+1} \times \dots \times K_{d+s}\big)\] 
can be written as a product of $d+1$ infinite compact spaces with at most one metrizable factor, so it suffices to prove that such a space cannot be a continuous image of a product of $d$ linearly ordered compact spaces.

 On the other hand, Avil\'es \cite{Avil09} proved that,  under the assumptions  of \ref{i:1}, there are at least $s+1$ separable factors $K_j$ . This result, combined with Marde\v{s}i\'c's result mentioned above, implies that if a product of two linearly ordered compact spaces can be mapped onto a product of three infinite compact spaces $K_1,~K_2,~K_3$, then at least one $K_j$ is metrizable \cite{Avil09,Mard15}.

In this paper we present a proof of  Conjecture \ref{i:1}. Our argument is based on properties of a certain
kind of dimension, called here the free dimension, that we introduce and investigate in Section \ref{fd}.
The basic idea standing behind that concept   is that it is a topological invariant which, in a sense, can detect how many nonmetrizable factors a given compact space has.
To outline our proof consider, for simplicity, the case $d=2$.
Infinite compact lines have free dimension $1$, and a product of two such spaces has free dimension $\le 2$.
Since the free dimension is not increased by continuous surjections,  
Conjecture \ref{i:1} follows immediately from our Theorem 
\ref{fdp:1}, stating that, in particular, if $K_1,K_2$ are nonmetrizable compacta and $K_3$ is an infinite compactum then the free dimension of $K_1\times K_2\times K_3$ is at least $3$. 
This reasoning works for an arbitrary value of $d$,  see Section \ref{fdp} for details.

Let us recall that a Boolean algebra is said to be an {\em interval algebra} if it is generated by a linearly ordered set of generators. Interval algebras correspond, via the Stone duality, to zerodimensional linearly ordered compacta. In particular, \ref{i:1} says that if $\fA_1,\fA_2$ are uncountable Boolean algebras and
the algebra $\fA_3$ is infinite then  the free product 
$\fA_1\otimes \fA_2\otimes \fA_3$ cannot be embedded into
a free product of two interval algebras. In Section \ref{ba} we present some generalizations of such a result ---we single out  a larger class of Boolean algebras whose Stone spaces
do not admit continuous surjections onto products of too many nonmetrizable compact spaces.
It turns out that our notion of free dimension is naturally connected with properties of Boolean algebras
generated by families not containing too many independent elements.
 
 Baur and Heindorf \cite{BaurHeindorf97} investigated the class of the so called initial chain algebras 
(a natural extension of the class of interval algebras).
Using the free dimension we were able to answer a question posed in \cite{BaurHeindorf97}:
We prove  that if $\fA$ and $\fB$ are infinite Boolean algebras and $\fA$ is uncountable, 
then the free product $\fA \otimes \fB$ cannot be embedded into an initial chain algebra. 

We wish to thank Antonio Avil\'es for very stimulating discussions concerning the subject of the paper and several valuable suggestions.

\section{Free dimension of compacta}\label{fd}

Given a compact space $K$, we shall write $\fcc(K)$ for the family of finite closed covers
of $K$. For any two finite covers $\cC_1$ and $\cC_2$ of $K$  we write $\C_1\prec \cC_2$ if the cover $\cC_1$ is finer than $\cC_2$, that is, if every $C_1\in \cC_1$ is contained in some $C_2\in\cC_2$.
Let us say that the family $\fC\sub\fcc(K)$ is a {\em topologically cofinal family of covers}
if for every {\bf open} cover $\cU$ of $K$ there is $\cC\in \fC$ such that $\cC\prec \cU$.
 
Our considerations are based on the following new concept.

\begin{definition}\label{fd:1}
Let $d$ be a natural number. We say that
 a compact space $K$ has {\em free dimension} $\le d$ and write $\fd(K)\le d$ if 
 there are a topologically cofinal family $\fC\sub \fcc(K)$, a constant $M>0$ and a function
 $\chi:\fC\to\en$ such that for every $k$, any
 $\C_1, \C_2, \dots, \C_k\in \fC$ have a joint refinement $\cC\in\fcc(K)$ such that
 \[\big| \cC\big|\le M\big(\chi(\cC_1)+\chi(\cC_2)+\ldots+ \chi(\cC_k)\big)^d.\]
 \end{definition}

We also write $\fd(K)\ge d+1$ to indicate that $\fd(K)\le d$ does not hold;
the relation $\fd(K)=d$ is defined accordingly.

\begin{lemma}\label{fd:1.5}
A compact space $K$ has free dimension $0$ if and only if it is finite.
\end{lemma}
\begin{proof}
If $K$ has free dimension $0$ then it is immediate that any constant $M$ witnessing that $\fd(K)=0$ must satisfy $|K| \le M$, so $K$ is finite.

Conversely, if $K$ is finite then the cover consisting of singletons is finer than any other cover, so  $\fC = \fcc (K)$, $M=|K|$ and any function $\chi:\fC \to \en$ satisfy the condition on Definition \ref{fd:1} for $d=0$.
\end{proof}

Let us note that we allow any constant $M$ in Definition \ref{fd:1} for the sake of Lemma \ref{fd:1.5};
if $d\ge 1$ then we can as well take $M=1$, after rescaling the function $\chi$ in question.
Thus in the sequel, we rather use the following definition.

\begin{remark}
\label{remark1}	
A compact space $K$ has free dimension $\le d$, where $d\ge 1$, if and only if there are a topologically cofinal family $\fC\sub \fcc(K)$ and a function
$\chi:\fC\to\en$ such that for every $k$, any
$\C_1, \C_2, \dots, \C_k\in \fC$ have a joint refinement $\cC\in\fcc(K)$ such that
\[\big| \cC\big|\le \big(\chi(\cC_1)+\chi(\cC_2)+\ldots+ \chi(\cC_k)\big)^d.\]
\end{remark}

One might expect that the function $\chi$ from Definition \ref{fd:1} should be of
some concrete type, such as  $\chi(\cC)=M\cdot |\cC|$ for some constant $M$. However,
our definition makes the notion of free dimension more flexible and simplifies some considerations,
cf.\ Theorem \ref{fd:5} below.

\begin{lemma}\label{fd:1.7} If $K$ is an infinite linearly ordered  compact space
then $\fd(K)=1$.
\end{lemma}

\begin{proof}
Indeed, let the family $\fC\sub\fcc(K)$ consists of all finite covers of $K$  by closed intervals.
Then $\fC$ is topologically cofinal. Let $\chi:\fC\to\en$ be defined as $\chi(\cC)=2|\cC|$.

If  $\cC_1,\ldots,\cC_k\in\fC$ then we define their joint refinement $\cC$ by listing all the endpoints
of intervals from $\cC_1\cup\ldots\cup \cC_k$ and taking intervals with consecutive endpoints.
Clearly,  $|\cC|\le \chi(\cC_1)+\ldots+ \chi(\cC_k)$.
\end{proof}

As we shall see soon, a product of $d$-many compact lines is a typical example of a compact space of free dimension $\le d$.

\begin{lemma}\label{fd:2}
If $K_i$ is a  compactum of free dimension $\le d_i$ for $i=1,2$ then
 \[\fd(K_1\times K_2)\le d_1+d_2.\]
\end{lemma}

\begin{proof}
We may suppose $K_1$ and $K_2$ are infinite; otherwise the result is trivial.	
Following Remark \ref{remark1}, take $\fC_i\sub\fcc(K_i)$ and a function $\chi_i:\fC_i \to \en$ witnessing that $\fd(K_i)\le d_i$.

For every $\cC_i\in\fC_i$, $i=1,2$, define
 \[\cC_{1,2}=\{ C_1\times C_2: C_i\in\cC_i\},\quad \chi (\cC_{1,2})=\max (\chi_1(\cC_1), \chi_2(\cC_2)).\] 
Let $\fC\sub \fcc(K_1\times K_2)$ be the family of all covers obtained in this way; we shall check 
that $\fC$ and the function $\chi : \fC \to \en$ are as required. 

We prove first that $\fC$ is topologically cofinal. Take $\cU$ any open cover of $K_1 \times K_2$. 
By compactness, there is an open cover $\cV$ of $K_1\times K_2$ made of open rectangles such that 
$\cV \prec \cU$. Moreover, $\cV$ can be taken of the form 
$\cV=\lbrace U \times V: U \in \cU_1 , V \in \cU_2 \rbrace$ with $\cU_i$ being an open cover of $K_i$. 
Since $\fC_i$ are topologically cofinal in the respective spaces, there are 
$\cC_i \in \fC_i$ such that $\cC_i \prec \cU_i$. Thus, the cover 
$\cC = \lbrace C_1 \times C_2 : C_i \in \cC_i \rbrace \in \fC$ satisfies $\cC \prec \cV \prec \cU$, so $\fC$ is indeed topologically cofinal in the product space.

Now fix $\cC_1, \cC_2, \dots , \cC_k \in \fC$.
For every $\cC_i$ take $\cC_i^1 \in \fC_1$, $\cC_i^2 \in \fC_2$ such that 
\[ \cC_i = \lbrace C_1 \times C_2 : C_j \in \cC_i^j,~j=1,2\rbrace,\]
 for every $i \leq k$.  By definition, for $j=1$ and $j=2$, the family
 $\cC_1^j, \cC_2^j, \dots , \cC_k^j$ has  a joint refinement $\cC'_j \in \fcc(K_j)$ such that
\[\big| \cC'_j \big|\le \left(\chi_j(\cC_1^j)+\chi_j(\cC_2^j)+\ldots+ \chi_j(\cC_k^j)\right)^{d_j} \le \left(\chi(\cC_1)+\chi(\cC_2)+\ldots+ \chi(\cC_k)\right)^{d_j}.\]

Thus, if we take $\cC = \lbrace C_1 \times C_2 : C_j \in \cC_j' \rbrace \in \fC$ then $\cC$ is a joint refinement of $\cC_1, \cC_2, \dots , \cC_k $ and 
\[\big| \cC \big| \le \big| \cC_1' \big| \big| \cC_2' \big| \le \left( \chi(\cC_1)+\chi(\cC_2)+\ldots+ \chi(\cC_k)\right)^{d_1+d_2},\]
so the proof is complete.
\end{proof}

 Lemma \ref{fd:1.7} and  Lemma \ref{fd:2} yield immediately the following.

\begin{corollary}\label{fd:3}
If $L_1,L_2,\ldots, L_d$ are linearly ordered compact spaces then
\[\fd\big(L_1\times L_2\times\ldots L_d\big)\le d.\] 
\end{corollary}

\begin{lemma}
\label{fd:4}	
Let $K$ be a compactum of free dimension $\le d$. Suppose $L$ is a  compact space such that there is a function $F:\fcc(K)\to \fcc(L)$ preserving topologically cofinal families and such that $F(\cC) \prec F(\cC')$ and $\big|F(\cC)\big| \le \big|\cC\big|$ whenever $\cC \prec \cC'$, $\cC,\cC' \in \fC$. Then 
$\fd(L)\le d$. 
\end{lemma}
\begin{proof}
Note that a compact space $K$ is finite if and only if there is a topologically cofinal family $\fC \sub \fcc(K)$ consisting of just one cover. Therefore, since $F$ preserves topologically cofinal families, if $K$ is finite then $L$ is also finite and both have free dimension $0$.

Suppose now $K$ and $L$ infinite. Take $\fC\sub\fcc(K)$ and $\chi:\fC \to \en$ witnessing that $\fd(K)\le d$. Set 
\[\fC' = F(\fC) = \lbrace F(\cC): \cC \in \fC \rbrace \sub \fcc(L).\]
Since $\fC$ is topologically cofinal, it follows from the hypothesis on $F$ that $\fC'$ is also topologically cofinal.
Let $\chi': \fC' \to \en$ be the function defined as $\chi'(\cC')=\min \lbrace \chi(\cC): \cC \in \fC, ~F(\cC)=\cC' \rbrace$ for every $\cC' \in \fC'$.

Take $\cC_1', \cC_2', \dots, \cC_k' \in \fC'$. Set $\cC_1, \cC_2, \dots, \cC_k \in \fC$ such that $F(\cC_i)=\cC_i'$ and $\chi'(\cC_i')=\chi(\cC_i)$ for every $i\le k$.
By definition, there is a joint refinement $\cC \in \fC$  of $\cC_1, \cC_2, \ldots, \cC_k$ such that 
\[\big| \cC\big|\le \big(\chi(\cC_1)+\chi(\cC_2)+\ldots+ \chi(\cC_k)\big)^d.\]
By the monotonicity of $F$, $F(\cC)$ is a joint refinement of $\cC_1', \cC_2', \dots, \cC_k'$ and
\[\big| F(\cC)\big|\le \big| \cC \big| \le  \big(\chi(\cC_1)+\chi(\cC_2)+\ldots+ \chi(\cC_k)\big)^d=\big(\chi'(\cC_1')+\chi'(\cC_2')+\ldots+ \chi'(\cC_k')\big)^d .\]
\end{proof}

\begin{theorem}\label{fd:5}
The class of compacta $K$ with $\fd(K)\le d$ is stable under taking closed subspaces
and continuous images.
\end{theorem}

\begin{proof}
If $g:K\to L$ is a continuous surjection then we apply the previous lemma with 
\[F: \fcc(K) \to \fcc(L), \mbox{ where } F(\cC)=\{g[C]: C\in\cC\} \mbox{  for every } \cC\in\fcc(K).\]
Notice that if $\cU$ is any open cover of $L$, then 
$\cU' = \lbrace  g^{-1}(U): U \in \cU \rbrace$ is an open cover of $K$ and $F(\cC) \prec \cU$ whenever $\cC \prec \cU'$. Thus, $F$ preserves topologically cofinal families and it is immediate that $F$ satisfies the rest of the hypotheses of Lemma \ref{fd:4}.

If $L \subseteq K$ is a closed subspace then the function $F: \fcc(K) \to \fcc(L)$ defined as $F(\cC)=\{C\cap L: C\in\cC\}$ for every $\cC\in\fcc(K)$ also satisfies the hypotheses of Lemma \ref{fd:4}. In particular, if $\cU$ is an open cover of $L$, then we can take an open cover $\cU'$ of $K$ such that $\cU=\lbrace  U\cap L: U \in \cU' \rbrace$ and in this case, $F(\cC) \prec \cU$ whenever $\cC \prec \cU'$.  
\end{proof}

\begin{corollary}
If $K$ is a metric compactum then $\fd(K)\le 1$.
\end{corollary}

\begin{proof}
We have $\fd(2^\omega)=1$; therefore  we conclude that $\fd(K)\le 1$ for every metric compactum $K$,
using Theorem \ref{fd:5} and the fact that $K$ is a continuous image of 
the Cantor set $2^\omega$.
\end{proof}

\section{Free dimension of products}\label{fdp}

The main result of the sections reads as follows.

\begin{theorem}\label{fdp:1}
Consider a space $K$ of the form
\[ K=K_1\times K_2\times\ldots\times K_d\times K_{d+1}, \]
where $d\ge 1$,  $K_1,\ldots, K_d$ are nonmetrizable compacta, while  $K_{d+1}$ is
an infinite compact space. Then $\fd(K)\ge d+1$.
\end{theorem}

We first prove the following auxiliary result.

\begin{lemma}\label{fdp:2}
Let $K$ be a nonmetrizable compactum.
There is an uncountable family $\cF$  of continuous functions $K\to [0,1]$ such that
for every uncountable $\cF'\sub\cF$ and infinite $\cF''\sub\cF'$, for every $n$
there are   functions $f_1,\ldots, f_n\in\cF''$ and  points 
$x_1,\ldots, x_{n+1}\in K$ such that for every $1\le j< j'\le n+1$ there is $k\le n$ such that
$|f_k(x_j)-f_k(x_{j'})| \ge 1/2$.
\end{lemma}

\begin{proof}
	Since $K$ is not metrizable, there is a family $\cF=\{f_\alpha:\alpha<\omega_1\}$ of continuous functions
	and points $x_\alpha^0, x_\alpha^1\in K$ such that
	
	\begin{enumerate}[(a)]
		\item $f_\alpha(x_\alpha^i)=i$ for every $\alpha<\omega_1$ and $i\in\{0,1\}$;
		\item $f_\beta(x_\alpha^0)= f_\beta(x_\alpha^1)$ for every $\beta<\alpha<\omega_1$.
	\end{enumerate}
	
	The inductive construction proving the above claim is straightforward: 
	Given a countable family $\cF_0=\{f_\xi:\xi< \alpha\}$, there are
	two points $x_\alpha^0, x_\alpha^1$ that are not separable by $\cF_0$ (as $K$ is not metrizable). We
	then choose $f_\alpha$ satisfying (a).

	It remains to  check that the family $\cF$ satisfies the assertion. Consider any uncountable $J\sub\omega_1$. Set 
	\[c_{\beta,\alpha}=f_\beta(x_\alpha^0)= f_\beta(x_\alpha^1),\]  
for every $\beta<\alpha<\omega_1$ and write $I_0=[0, 1/2]$, $I_1=(1/2,1]$. 

Take any infinite $N\sub J$. For $\beta,\alpha\in N$, $\beta<\alpha$,
say that $\{\beta, \alpha\}$ gets color $0$ if $c_{\beta,\alpha}\in I_0$, and
gets color $1$ otherwise. By the Ramsey theorem,
there is an infinite  homogeneous set $N_1\sub N$. Let us suppose that $N_1$
is $0-$homogeneous, the other case is analogous.

Fix any natural number $n$. We claim that the functions $f_{\alpha_1}, \dots, f_{\alpha_n}$,
where $\alpha_i\in N_1$, $\alpha_{1}<\ldots<\alpha_n$,    and the points 
$x_{\alpha_1}^1, x_{\alpha_2}^1,\dots , x_{\alpha_n}^1, x_{\alpha_n}^0$ 
satistify the assertion of the lemma.
	
	Notice that 
\[|f_{\alpha_i}(x_{\alpha_i}^1)-f_{\alpha_i}(x_{\alpha_j}^\epsilon)|=|1-c_{\alpha_i, \alpha_{j}}| \geq 1/2,\] 
whenever $i<j\leq n$ and $\epsilon \in \lbrace 0,1 \rbrace$. Since $|f_{\alpha_n}(x_{\alpha_n}^1)-f_{\alpha_n}(x_{\alpha_n}^0)|=1$, the proof is complete.
\end{proof}

\begin{proof}(of Theorem \ref{fdp:1})
We argue by contradiction: suppose that 
$\fC\sub \fcc(K)$ is a topologically cofinal family of finite closed covers witnessing, together
with a function $\chi:\fC\to\en$,  that $\fd(K)\le d$.

For simplicity, say that a closed finite cover $\cC$  of $K$ is {\em good} for a function
$f\in C(K)$ if the oscillation of $f$ is $\le 1/3$ on every $C\in\cC$. Note that, since
$\fC$ is topologically cofinal, for every $f\in C(K)$ there is $\cC\in\fC$ which is
good for $f$.
For $i\le d+1$ we write $\pi_i:K\to K_i$ for the projection.
In particular,  for any $i\le d+1$ and $g\in C(K_i)$, there is $\cC\in\fC$ which is good
for $g\circ\pi_i$.

For every $i\le d$ take an uncountable  family $\cF^i\sub C(K_i)$,  
as in Lemma \ref{fdp:2}. Since the assertion of \ref{fdp:2} allows us to pass
to uncountable subfamilies, we can and do assume that there is a natural number $m$
such that for every $i\le d$ and $f\in\cF^i$ there exists a partition $\cC\in \fC$ 
which is good for $f\circ\pi_i$ and satisfies $\chi(\cC)\le m$.

Since the space $K_{d+1}$ is infinite we can choose a sequence 
$h_i: K_{d+1}\to [0,1]$ of pairwise disjoint norm-one continuous functions.
Fix any $p\ge (m \cdot d)^d$.

Let $m_1$ be a natural number such that for every $i\le p$, the  
function $h_i\circ\pi_{d+1}$ admits a good cover $\cC\in\fC$ satisfying $\chi(\C)\le m_1$.

We now consider some large number $n$; we shall specify it later. 
For every $i\le d$ take $\{g^i_1,\ldots, g^i_n\}\sub \cF^i$ satisfying the assertion
of Lemma \ref{fdp:2}. Put 
\[\cG^i=\{g^i_j\circ\pi_i: j\le n\},  \quad \cH=\{h_1\circ\pi_{d+1},\ldots, h_p\circ\pi_{d+1}\},\]
and write $\cG=\cG^1\cup\ldots\cup \cG^d$.

For every function $f\in \cG$ there is a cover $\cC_f\in\fC$ such that $\chi(\cC_f)\le m$,
which is  good for $f$.
Likewise, for $f\in\cH$ there is a cover $\cC_f\in\fC$ such that $\chi(\cC_f)\le m_1$, which is  good for $f$.
By the definition of the free dimension, there is a finite closed cover $\cC$ of $K$, which refines all $\cC_f$ for $f\in \cG\cup \cH$, and satisfies
\[ |\cC|\le \left( \sum_{f\in \cG}\chi(\cC_f) +\sum_{f\in \cH} \chi(\cC_f)\right)^d\le  \left(n\cdot d \cdot m+ p\cdot m_1\right)^d.\]
Note that $\cC$ is good for every function from $\cG\cup\cH$. To summarize,
\[ (*) \mbox{ there is a cover of } K \mbox{ of size } (n\cdot d \cdot m+ p\cdot m_1)^d
\mbox{ which is good for every } f\in \cG\cup\cH.\]

On the other hand, for every $i\le d$ there are points $x^i_1,\ldots, x^i_{n+1}\in K_i$ as in 
Lemma \ref{fdp:2} (for the family $\{g^i_1,\ldots, g^i_n\}$). In the space
$K_{d+1}$ choose points $y_i$, $i\le p+1$ so that $h_i(y_i)=1$ for $i\le p$ and
$h_i(y_{p+1})=0$ for $i\le p$. 

Consider now the set $D\sub K$, where
\[D=\{x^1_1,\ldots, x^1_{n+1} \}\times \ldots\times \{x^d_1,\ldots, x^d_{n+1}\}\times \{y_1,\ldots, y_{p+1}\}.\] 
Note that $|D|=(n+1)^d(p+1)$ and 
\[ (**) \mbox{ for } d,d'\in D, \mbox{ if } d\neq d' \mbox{ then there is } f\in \cG\cup\cH
\mbox{ such that } |f(d)-f(d')|\ge 1/2.\]
Indeed, if, for instance, $d,d'\in D$ and $d_1\neq d_1'$ then we use
the property granted by Lemma \ref{fdp:2} to find a function $f^1_i$ such that
$|f^1_i(d_1)-f^1_i(d_1')|\ge 1/2$; then 
\[ |f^1_i\circ\pi_1(d)-f^1_i\circ\pi_1 (d')|\ge 1/2.\]

Clearly, (**) implies that there is no cover of $K$ of size $< (n+1)^d(p+1)$ which is good for every function from  $\cG\cup\cH$.

Finally, (*) and (**) yield a contradiction, whenever $n$ satisfies
\[ (n\cdot d \cdot m+ p\cdot m_1)^d< (n+1)^d(p+1),\]
which eventually holds as $p+1>  (m \cdot d) ^d$.
\end{proof}

\begin{corollary}\label{fdp:3}
The conjecture \ref{i:1} holds true.
\end{corollary}

\begin{proof}
As we explained in the introduction,  it is enough to check that
there is no continuous surjection
\[ L=L_1\times L_2\times\ldots L_d\to K=K_1\times\ldots K_d\times K_{d+1},\]
whenever $L_1,\ldots, L_d$ are compact lines, $K_1,\ldots K_d$ are nonmetrizable compacta
and $K_{d+1}$ is an infinite compact space.

This follows from  Lemma \ref{fd:5}, as $\fd(L)\le d$ by Lemma \ref{fd:3} and
$\fd(K)\ge d+1$ by Theorem \ref{fdp:1}.
\end{proof}

\section{Free dimension and Boolean algebras}\label{ba}

The purpose of this final  section is to discuss some classes of  Boolean algebras whose Stone spaces have free dimension $\leq d$.

Let $\fA$ be a Boolean algebra; we denote by $K_\fA$ its Stone space of all ultrafilters on $\fA$.
Given $a\in\fA$, $\widehat{a}=\{x\in K_\fA: a\in x\}$ is the corresponding clopen subset of $K_\fA$.
 
For  every finite set $F\sub \fA$, we write $\At(F)$ for the family of atoms of the Boolean subalgebra generated by $F$. 
Note that every family $\At(F)$ defines  a finite closed cover 
\[ \cC_F=\{\widehat{a}: a\in\At(F)\},\]
of $K_\fA$. Moreover, if $\Gamma\sub\fA$ generates $\fA$ then 
\[ \fC = \{ \cC_F: F\sub \Gamma \mbox{ finite}\} \sub \fcc(K_\fA),\]
 is a topologically cofinal family of covers of  $K_\fA$. 
 
 A natural class of Boolean algebras whose Stone spaces 
 have free dimension $\le d$
  is described by the following result.

\begin{theorem} \label{ba:1}
Let $d$ be a natural number and $M\ge0$. Suppose that a Boolean algebra $\fA$ is generated by a set $\Gamma$ such that every finite set $F \subseteq \Gamma$ satisfies 
$ \big| \At(F)\big|  \leq M |F|^d$.

 Then $\fd(K_\fA)\le d$.	
\end{theorem}

\begin{proof}
As before, set $\fC = \{ \cC_F: F\sub \Gamma \mbox{ finite}\} \sub \fcc(K_\fA)$.

Take any finite $F\sub\Gamma$. To deal with $\At(F)$ we can assume that $F$
is irredundant, that is, no $g\in F$ belongs to the algebra generated by $F\setminus \{g\}$.
It is easy to check that if $F$ is irredundant then $|F|\le \big|\At(F)\big|$.

It is now sufficient to take irredundant finite sets $F_1,F_2, \ldots, F_k \sub\Gamma$ and observe that 
$\cC_F$ is a joint refinement of $\cC_{F_1}, \ldots, \cC_{F_k}$, where $F=F_1\cup \ldots \cup F_k$. Moreover,
\[ \big| \At(F) \big|\le M|F|^d\le  M(|F_1|+ \ldots +|F_k|)^d\le M( |\At(F_1)|+\ldots+|\At(F_k)|)^d,\]
which gives the result.
\end{proof}

A  Boolean algebra $\fA$ is an {\em  interval algebra}  if $\fA$ is generated by 
a subset $\Gamma\sub\fA$ which is linearly ordered.
Heindorf \cite{He97} proved that a Boolean algebra $\fA$ embeds into an interval algebra 
if and only if it is generated by a set $\Gamma$ with the property that any two elements of $\Gamma$ are either comparable or disjoint (so,  in particular,  $\Gamma$ contains no independent pair).
 There is a larger class $\II(d)$, introduced in \cite{AMCP}, of Boolean algebras satisfying 
the condition of Theorem \ref{ba:1}. 

\begin{definition}\label{ba:2}
Let $d \in \N$. A Boolean algebra $\fA$ is said to be in the class $\II(d)$ if $\fA$ is generated by a set $\Gamma$ with the property that no $d+1$ elements in $\Gamma$ are independent.
\end{definition}

Here we use the usual notion of Boolean independence; 
recall that $a_1, \ldots, a_d \in \fA$ are said to be independent if 
\[ \bigcap_{i \in S_1} a_i \cap \bigcap_{j \in S_2} a_j^c \neq {\bf 0},\]  
for every two disjoint sets 
$S_1, S_2 \sub \{ 1,2, \ldots, d\}$.

It is easy to check that the class $\II(d)$ contains all the $d$-fold free products of interval algebras.
The fact that every Boolean algebra in the class $\II(d)$ satisfies the condition of Theorem \ref{ba:1} is a consequence of the following Sauer-Shelah Lemma; see e.g.\ \cite{AMCP} for a proof and further references.

\begin{lemma}[Sauer-Shelah]\label{ssl}
	Let $N,d$ be natural numbers with $0\le d < N$ and let $T=\lbrace 1,2,\dots,N \rbrace$.
	Then for every family $C \sub 2^T$ with
	\[|C|> \binom{N}{0}+ \binom{N}{1}+ \dots + \binom{N}{d},\]
	there exists a set $S \sub T$ with $|S|=d+1$ such that $ \lbrace f|_S : f \in C \rbrace = 2^S.$
\end{lemma}

\begin{lemma}\label{ba:3}
	Suppose $\fA \in \II(d)$ is generated by a set $\Gamma$ with no $d+1$ independent elements.
	Then for every finite set $F\subseteq \Gamma $ we have 
\[ \big| \At(F)\big| \leq \binom{|F|}{0}+ \binom{|F|}{1}+ \dots + \binom{|F|}{d} \leq (d+1)|F|^d.\] 
\end{lemma}

\begin{proof}
Set $F=\lbrace a_1, a_2, \dots, a_N \rbrace$ and $T=\lbrace 1,2,\dots,N \rbrace$. Then, every element in $\At(F)$ has a unique representation of the form $a_1 ^{f(1)} \cap a_2 ^{f(2)} \cap \dots \cap a_N ^{f(N)}$, where $f \in 2^T$, $a^0=a^c$ and $a^1=a$ for every $a\in \fA$. Define
\[ C= \lbrace f \in 2^T: a_1 ^{f(1)} \cap a_2 ^{f(2)} \cap \dots \cap a_N ^{f(N)}  \in \At(F) \rbrace .\]

It follows from the fact that $\Gamma$ contains no $d+1$ independent elements and Sauer-Shelah Lemma that 
\[  \big| \At(F)\big| =|C| \leq \binom{N}{0}+ \binom{N}{1}+ \dots + \binom{N}{d} =\]
\[ =\binom{|F|}{0}+ \binom{|F|}{1}+ \dots + \binom{|F|}{d}\leq (d+1)|F|^d,\]

since $\binom{|F|}{i} \leq |F|^d$ for every $i \leq d$.
\end{proof}

An immediate consequence of Theorem \ref{ba:1} and Lemma \ref{ba:3} is the following result.

\begin{corollary}\label{ba:4}
For every Boolean algebra  $\fA \in \II(d)$, $\fd(K_\fA)\le d$.

\end{corollary}

We shall finally discuss free dimension of Stone spaces of initial chain algebras;
Corollary \ref{ba:final} below answers a question posed in
\cite{BaurHeindorf97}.

A pseudotree is a partially ordered set $(T, \le)$ in which every set of the form $(-\infty, t]:=\lbrace s \in T: s \leq t \rbrace $ is linearly ordered. S.\ Koppelberg and J.D.\ Monk defined the \textit{initial chain algebra on a pseudotree $T$} to be  the algebra of subsets of $T$ generated by the initial chains $(-\infty, t]$ with $t\in T$.

\begin{lemma}\label{ba:5}
Let $\fA$ be an initial chain algebra on a pseudotree $T$ and $F=\lbrace (-\infty, t_i]: i=1, \dots, n  \rbrace$  a family of initial chains on $T$. Then $\At(F) \leq 2 |F|$.
\end{lemma}

\begin{proof}
Without loss of generality, we suppose that $T$ is a well-met tree (i.e. every two uncompatible elements have a greatest lower bound) so every element of $\fA$ is of the form 
$b=\bigcup_{s=1}^n(s_i,t_i]$ or $b=T\setminus \bigcup_{s=1}^n(s_i,t_i]$ for some $s_i, t_i \in T$ (see \cite[Corollary 1.3 and Lemma 1.4]{BaurHeindorf97}). 
We claim that there is a set of points $P$ such that every element in $\At(F)$, except at most one, is of the form $(s,s']$ with $s,s' \in P$ and such that $(-\infty, s]$ is in the algebra generated by $ F $ for every $s \in P$. We prove the lemma and the claim by induction on $|F|$. If $F=\lbrace (-\infty, t] \rbrace$, then the result is trivial since the atoms are $(-\infty, t]$ and $(-\infty, t]^c$. Suppose the result true for $n$ and take $F=\lbrace (-\infty, t_i]: i=1, \dots, n  \rbrace$. Set $t \in T$ and $F'= F \bigcup \lbrace (-\infty, t] \rbrace$.  We claim that $(-\infty,t]$ splits at most one atom of the form $(s_1,s_2]$. By contradiction, suppose $(-\infty,t]$ splits two atoms of the form $(s_1,s_2], (s_1',s_2']$ with $s_1,s_2,s_1',s_2' \in P$. Then take $t_s = \min (s_2, t)$ and $t_{s'}= \min (s_2', t)$. Notice that $s_1 < t_s \leq t$ and $s_1' < t_{s'} \leq t$. Since $(-\infty, t]$ is linearly ordered, without loss of generality we may suppose $s_1 \leq s_1'$. Since the intervals $(s_1,s_2]$ and $(s_1',s_2']$ are disjoint, we have that $t_s \notin (s_1',s_2']$, so 
$$ s_1 < t_s \leq s_1' <t .$$ But then $(s_1, s_1']=(-\infty, s_1']\cap (-\infty,s_1]^c $ belongs to the algebra generated by $F$ and it intersects $(s_1, s_2] \in \At(F)$, so $(s_1,s_2] \cap (s_1, s_1']= (s_1, s_2]$ and $s_2 \le s_1'$. But then
$$ s_1 < t_s \le s_2 \le s_1' < t_{s'} \le t$$
and therefore $t_s = \min (s_2, t)=s_2$. This contradicts the fact that $(-\infty, t]$ splits $(s_1, s_2]$.

Thus, $(-\infty,t]$ splits at most one atom of the form $(s_1,s_2]$ and 
$$ \At(F') \leq \At(F)+2 \leq 2|F|+2=2(|F|+1)=2|F'|,$$
as we wanted to prove. 
\end{proof}

\begin{corollary}\label{ba:final}
For every initial chain algebra $\fA$, $\fd(K_\fA)\leq 1$.
In particular, $\fA$ does not contain a free product of an infinite Boolean algebra and an uncountable Boolean algebra.
\end{corollary}
\begin{proof}
It follows from Theorem \ref{ba:1} and Lemma \ref{ba:5} that $\fd(K_\fA)\leq 1$.
The last assertion is a consequence of Theorem \ref{fdp:1}, since the Stone space of every uncountable Boolean algebra is nonmetrizable. 
\end{proof}

\end{document}